\documentclass[12pt]{article}
\usepackage[margin=1in]{geometry}  
\usepackage{amsmath,amsthm,amssymb,amsfonts}

\usepackage{bbm} 

\usepackage[dvipsnames]{xcolor}   

\usepackage[utf8]{inputenc}   

\usepackage{hyperref}       

\usepackage{enumitem}           

\usepackage{soul}               

\usepackage{easyReview}   

\usepackage[square,numbers]{natbib}

\newcommand{\Prob}{\ensuremath{\mathbb{P}}}
\newcommand{\E}{\ensuremath{\mathbb{E}}}

\newcommand{\C}{\ensuremath{\text{Cov}}}

\newtheorem{theorem}{Theorem}
\newtheorem{corollary}{Corollary}[theorem]

\theoremstyle{definition}
\newtheorem{proposition}{Proposition}[section]
\newtheorem{example}{Example}[section]

\theoremstyle{remark}
\newtheorem{remark}{Remark}[section]

\allowdisplaybreaks

\begin{document}
	\title{\textbf{Stopping Times Occurring Simultaneously}}
	\author{
	Philip Protter\thanks{Supported in part by NSF grant DMS-2106433} \\
	\small{Department of Statistics}\\
	\small{Columbia University}\\
	\small{New York, NY, 10027}\\
	\small{ORCID 0000-0003-1344-0403}\\
	\small{pep2117@columbia.edu}
	\and
	Alejandra Quintos\thanks{Supported in part  by the Office of the Vice Chancellor for Research and Graduate Education at the University of Wisconsin-Madison with funding from the Wisconsin Alumni Research Foundation and by the Fulbright-García Robles Program} \thanks{Corresponding author} \\
	\small{Department of Statistics}\\
	\small{University of Wisconsin-Madison}\\
	\small{Madison, WI, 53706}\\
	\small{ORCID 0000-0003-3447-3255}\\
	\small{alejandra.quintos@wisc.edu}
}


\date{April 15, 2024}
\maketitle

\section*{ABSTRACT}
Stopping times are used in applications to model random arrivals. A standard assumption in many models is that they are conditionally independent, given an underlying filtration. This is a widely useful assumption, but there are circumstances where it seems to be unnecessarily strong. We use a modified Cox construction along with the bivariate exponential introduced by Marshall and Olkin (1967) to create a family of stopping times, which are not necessarily conditionally independent, allowing for a positive probability for them to be equal. We show that our initial construction only allows for positive dependence between stopping times, but we also propose a joint distribution that allows for negative dependence while preserving the property of non-zero probability of equality. We indicate applications to modeling COVID-19 contagion (and epidemics in general), civil engineering, and to credit risk.

\pagebreak 

\section{INTRODUCTION}

Probability models are ubiquitous in modern society.  The timing of a random event is often crucial to the analysis of the reliability of a system or to the danger of a default within a credit risk context. Such random times are, of course, referred to as stopping times. Some examples of random times of intrinsic interest range from the quotidian banal such as bus arrivals, customers arriving at a restaurant, to the less banal, such as the time a cancer metastasizes, or an individual contracts a contagious disease such as COVID-19. Stopping times appear in Civil Engineering as they model metal fatigue in aircraft, the times of collapse of a bridge, or extreme events such as the recent collapse of the condominium towers in Surfside Florida. A particularly common use of stopping times is in the theory of Credit Risk, within the discipline of Mathematical Finance. 

Indeed, we take the approach pioneered by researchers in Credit Risk, the seminal event being the publication of the book by David Lando in 2004 (see \cite{Lando.Credit.Risk.Book}). The ``Cox Construction" that Lando uses begins with a filtration of observable events, satisfying the usual hypotheses (see Protter \cite{PhilipBook}) for the formal definition of the usual hypotheses). To fix notation, let $\left(\Omega, \mathcal{F}, \Prob, (\mathcal{F}_t)_{t\geq 0}\right)$  be the probability space on which $\left(\mathcal{F}_t\right)_{t\geq 0}$ is the filtration of observable events. We will add a random time $\tau$ to our model by creating an exponential random variable $Z$ (of parameter 1) that is independent of $ \mathcal{F}$ and all of $\mathbb{F}=(\mathcal{F}_t)_{t\geq 0}$. We choose an $\mathcal{F}_t$ predictable increasing process $\left(A_t\right)_{t\geq 0}$, and we create the desired random time $\tau$ by writing:
\begin{equation}\label{ph1}
	\tau = \inf_{t\geq 0}\{A_t\geq Z\}. 
\end{equation}
$\tau$ is then a totally inaccessible stopping time for an enlarged (by $Z$) filtration $\mathbb{G}=(\mathcal{G}_t)_{t\geq 0}$, this is
\begin{equation*}
	\mathcal{G}_t=\mathcal{F}_t\vee\sigma(Z).
\end{equation*}
With $\tau$ being a stopping time, the process $1_{\{\tau\geq t\}}$ is adapted and non-decreasing, hence it is a submartingale, whence there exists, via the Doob-Meyer decomposition theorem, a unique, $\mathbb{G}$-predictable, increasing process $(C_t)_{t\geq 0}$ such that $1_{\{\tau\geq t\}}-C_t$ is a martingale. It is easy to show that the process $\left(C_t\right)_{t\geq 0}$ is in fact the process $\left(A_t\right)_{t\geq 0}$ of \eqref{ph1}. The process $\left(A_t\right)_{t\geq 0}$ is called the compensator of the stopping time $\tau$.

If we repeat this procedure to construct two stopping time $\tau_1$ and $\tau_2$, using two independent exponentials of parameter 1, $Z_1$ and $Z_2$, which are independent of each other as well as the underlying filtration $\mathbb{F}$, then we have two totally inaccessible stopping times which are also conditionally independent of each other, and we have $\Prob(\tau_1=\tau_2)=0.$ This is the approach of Lando and many others, with some notable exceptions such as Bielecki et al. \cite{Bielecki.Cousin.Crepey.Herbertsson} and Jiao and Li \cite{Jiao.Li}. 

In applications, often the process $\left(A_t\right)_{t\geq 0}$ is assumed to be of the form 
\begin{equation}\label{ph2}
	A_t=\int_0^t\alpha_s ds, \text{ for an adapted process }\alpha_s\geq 0\text{ all }s\geq 0.
\end{equation}
Sufficient conditions for $A_t$ to be of the form (\ref{ph2}) are known (see, e.g., Ethier and Kurtz \cite{Ethier.Kurtz}, Guo and Zeng \cite{Guo.Zeng}, Zeng \cite{Zeng}, and for a general result, Janson et al. \cite{Janson.Baye.Protter}).

In this paper we are concerned with models where one can have $\Prob(\tau_1=\tau_2)>0$, and some of the ramifications of such a model. This model arises in applications when $\eta_1$ and $\eta_2$ are constructed as functionals of Cox processes with independent exponentials $Z_1$ and $Z_2$, but with an added complication: there is a third stopping time $\eta_3$, and $\tau_1=\eta_1\wedge\eta_3$, while $\tau_2=\eta_2\wedge\eta_3$. This is a natural situation in Credit Risk for example, as we indicate in Section 3, but also in other domains, such as Civil Engineering, and disease contagion which we treat in Section \ref{Sect.Application}. The resulting stopping times of interest, $\tau_1$ and $\tau_2$, are no longer conditionally independent, and in simple cases, the bivariate exponential distribution of Marshall and Olkin \cite{MarshallOlkin} comes into play. These models do not have densities in $\mathbb{R}_+\times\mathbb{R}_+$, leading to a two dimensional cumulative distribution function with a singular component. We explore the consequences of such a phenomenon in some detail, and we explain its utility for various kinds of applications caused by the confluence of stopping times that arise naturally  in the modeling of random events.

One could argue that it is not of vital importance to have two stopping times with a positive probability of being equal, but rather just to have them be close to each other, even arbitrarily close. We study this situation in Section \ref{Subsect.Distance.Between.ST}.

In many easy to imagine examples, the times $\tau_1$ and $\tau_2$ are positively correlated. It is possible to imagine, however, situations where they would naturally be negatively correlated. To cover that situation, we slightly modify our constructions, as exhibited in Section \ref{Sect.A.More.General.Distribution}. As an example, consider the recent scandal with the Boeing 737 Max airplanes. The two horrifically deadly crashes occurred with the confluence of the airplane's software malfunctioning and the panic of inexperienced and poorly trained pilots. Since airplanes are so carefully constructed, software failures are rare, which made the Boeing 737 Max failures unanticipated; however, unfortunately, poorly trained pilots have recently become rather common, especially in airlines based in poorly regulated countries. (The two examples were Lion Air in Malaysia and Ethiopian Airlines, in Ethiopia). That duality between those ``rare" and ``common" events is due to their negative correlation.

Our work in this paper falls into a popular thread of prior research. Cox and Lewis \cite{CoxLewis} studied the case of multiple event types occurring in a continuum, but differently from us, they do not generalize the model proposed by Marshall and Olkin \cite{MarshallOlkin}. Moreover, they mostly consider ``regular" processes, i.e., where events cannot occur at the same time, which is an important contribution of our work. Diggle and Milne \cite{Diggle.Milne} also proposed a multivariate version of the Cox Process but their model does not allow for $\Prob(\tau_1=\tau_2)>0$ either. Another example of a multivariate point process is given in Brown et al. \cite{Brown.Silverman.Milne}.

There is existing work on modeling simultaneous defaults, but different from us, under Merton's structural risk model (see Li \cite{Weiping} and Bielecki et al. \cite{Bielecki2018}). Kay Giesecke \cite{Giesecke74}, in a seminal paper concerning Credit Risk published in 2003, was the first (to our knowledge) to consider the Marshall and Olkin model of the bivariate exponential. Later, in 2013, Bielecki et al. \cite{Bielecki.Cousin.Crepey.Herbertsson} also worked in a similar model, which is also discussed in Chapter 8 of Crépey et al. \cite{Crepey.Bielecki.Brigo.Book}. Sun et al. \cite{Sun_Mendoza-Arriaga_Linetsky_2017} in 2017 developed a similar model to the one we present here, but restricted their attention to multidimensional Lévy subordinator processes. In this paper, we develop the ideas present in Giesecke \cite{Giesecke74}, in Bielecki et al. \cite{Bielecki.Cousin.Crepey.Herbertsson}, and in Crépey et al. \cite{Crepey.Bielecki.Brigo.Book}, and go  beyond them.

Brigo et al. \cite{Brigo2016} argued that a reasonable trade-off between applicability and model flexibility might consist in modeling a multivariate survival indicator process such that each subvector of it is a continuous-time Markov chain, which, as they show, is a property of the Marshall–Olkin distribution. Hence, they claimed that the Marshall–Olkin distribution is the natural and unique choice for modeling default times. This strengthens our choice of joint distribution of the stopping times.

Another closely related work to ours is Lindskog and McNeil \cite{Lindskog.McNeil} who consider a Poisson shock model of arbitrary dimension with both fatal and not-necessarily-fatal shocks. Jiao and Li \cite{Jiao.Li} also allow for simultaneous defaults without using a pure Cox process construction. Inspired by the Jacod Criterion, they use what is known as a density approach, and they can include cases where a stopping time meets another stopping time, specified in advance.

There are also other types of generalizations of Cox Processes (see Gueye and Jeanblanc \cite{Gueye.Jeanblanc}), where they generalize, not the number of stopping times, but the form of the process $A_t$. They assume that $A_t$ is not necessarily of the form given in equation (\ref{ph2}).

The organization of this paper is as follows. Section \ref{Sect.Survival.Function}, presents the survival function of two conditionally dependent (as opposed to conditionally independent) stopping times, an interpretation of it, a decomposition of it into its singular and absolutely continuous parts, and a series of results exploring the special properties such a modeling approach has. For example, not only do we treat the case where $\Prob(\tau_1=\tau_2)>0$, but more generally we study when the two stopping times are `close' to each other, in various metrics. Section \ref{Sect.Generalization} provides two generalizations of our model. In the first one, we extend to an arbitrary (but finite) number of such stopping times and in the second one, we allow for a slightly different dependency between the stopping times. Section \ref{Sect.Application} shows the applicability of our results by providing examples in Epidemiology (such as the case of COVID-19 and its variants) and in Civil Engineering (e.g., the recent condo collapse of Champlain Towers in Florida).

\section{THE SURVIVAL FUNCTION} \label{Sect.Survival.Function}
As a starting point, we will consider the case of having two stopping times, which we will generalize to $n$ stopping times in Section \ref{Sect.Generalization.K.ST}. Let us fix a filtered probability space $(\Omega,\mathcal{F},\Prob,\mathbb{F})$ large enough to support an $\mathbb{R}^d$-valued càdlàg stochastic process $X=\left\{X_t, t \geq 0\right\}$ and three i.i.d exponential random variables with parameter 1 ($Z_i, \ i=1,2,3$) and independent of $X$. Then, define:

\begin{equation} \label{Main.Model}
	\tau_1 := \min \left( \eta_1, \eta_3 \right) \quad \textnormal{and} \quad \tau_2 := \min \left( \eta_2, \eta_3 \right)
\end{equation}
where
\begin{equation} \label{Initial.Definitions}
	\eta_i:=\inf\left\{ s: A^i_s \geq Z_i \right\}, \qquad A^i_s:=\int_0^s \alpha^i(X_r) dr, \qquad Z_i \overset{\text{iid}}{\sim}\text{Exp}(1)
\end{equation}
and
\begin{equation*}
	\alpha^i: \mathbb{R}^d \mapsto [0, \infty) \text{ for } i=1,2,3 \text{ are non-random positive continuous functions.}
\end{equation*}
Note that this characterization of $\alpha^i$ implies that $A^i_s$ are continuous and strictly increasing.
\begin{theorem} [Survival function] \label{Thm.Survival Fn tau1, tau2}
	Using the definitions above and assuming that $\lim_{s\rightarrow\infty}A_s=\infty$ a.s. and that $A^i_s < \infty \ \forall s \geq 0$. Then,
	\begin{equation*} 
		\overline{F}_{(\tau_1, \tau_2)}(s,t):= \Prob\left( \tau_1 > s, \tau_2 >t\right) = \E \left[\exp\left[ -A^1_s -A^2_t - A^3_{s \vee t} \right] \right] .
	\end{equation*}
\end{theorem}
\begin{proof}
	Let $\tilde{\Prob}_{(s,t)}(\cdot):= \Prob\left(\cdot |(X_u)_{0\leq u \leq (s \vee t)} \right)$. By definition of $\tau_1$ and $\tau_2$, we have,
	\begin{align*} 
		\tilde{\Prob}_{(s,t)}\left( \tau_1 > s, \tau_2 >t\right) & =  \tilde{\Prob}_{(s,t)} \left( \min(\eta_1, \eta_3) > s, \min(\eta_2, \eta_3) >t \right) \nonumber \\
		& = \tilde{\Prob}_{(s,t)} \left( \eta_1>s \right) \tilde{\Prob}_{(s,t)} \left( \eta_2>t \right) \tilde{\Prob}_{(s,t)} \left( \eta_3>s, \eta_3>t \right) \nonumber \\
		& = \tilde{\Prob}_{(s,t)} \left( \eta_1>s \right) \tilde{\Prob}_{(s,t)} \left( \eta_2>t \right) \tilde{\Prob}_{(s,t)} \left( \eta_3>\max(s,t) \right) \nonumber \\
		& = \exp\left[ -A^1_s -A^2_t - A^3_{s \vee t} \right] \nonumber \\
		& = \begin{cases}
			\exp\left[ -A^1_s -A^2_t - A^3_s  \right] & s>t \\
			\exp\left[ -A^1_s -A^2_t - A^3_t \right] & s<t.
		\end{cases}
	\end{align*}
	The result follows by taking the expectation.
\end{proof}

\begin{remark}
	From Theorem \ref{Thm.Survival Fn tau1, tau2}, it is clear that $\tau_1$ is not independent of $\tau_2$ and, as we will show in Subsection (\ref{Decomposition.Joint.Dist}), their joint distribution has an absolutely continuous and a non-trivial singular part, which means that:
	\begin{equation}
		\Prob\left( \tau_1 = \tau_2 \right) >0.
	\end{equation}
\end{remark}

\begin{remark}
	Note that when $X_t = t \in \mathbb{R}^+  $ (i.e., there is no randomness coming from $X_t$), Theorem \ref{Thm.Survival Fn tau1, tau2} is a generalization of the bivariate exponential (BVE) distribution introduced by Marshall and Olkin \cite{MarshallOlkin}. In the specific case where $\alpha^i(X_s)=\alpha^i\in \mathbb{R}$. i.e., the intensity is constant, we recover the Marshall and Olkin BVE with parameter $(\alpha^1, \alpha^2, \alpha^3)$. This is $(\tau_1, \tau_2)\sim\text{BVE}(\alpha^1, \alpha^2, \alpha^3)$.
\end{remark}

\begin{remark}
	The marginal distribution of $\tau_i$ is given by:
	\begin{align} \label{Marginal.Survival.Function}
		\Prob& (\tau_i>s)  = \Prob(\min(\eta_i, \eta_3) > s)= \Prob \left( \eta_i >s, \eta_3>s \right) = \E \left[ \Prob\left( \eta_i > s, \eta_3> s \mid \left( X_u \right)_{u\leq s} \right) \right] \nonumber \\
		&= \E \left[ \Prob\left( \eta_i > s \mid \left( X_u \right)_{u\leq s} \right) \Prob\left( \eta_3 > s \mid \left( X_u \right)_{u\leq s} \right) \right] = \E \left[\exp\left(-\int_0^s (\alpha^i+\alpha^3)(r)dr\right) \right] 
	\end{align}
	which coincides with the marginal distribution in the Cox construction, see Lando \cite{Lando.Cox} \cite{Lando.Credit.Risk.Book}. That is, each of the stopping times is functional of a Cox process. However, by construction, they are not independent of each other.
\end{remark}

\begin{remark}
	$\alpha^i()$ is usually known as the intensity of $\eta_i$ in the progressive enlargement of $\mathbb{F}$ with $\eta_i$. (See Corollary 2.26 in Aksamit and Jeanblanc \cite{Aksamit.Jeanblanc}).
\end{remark}

\begin{remark} \label{Intensity.Tau.i}
	The intensity of $\tau_i$ for $i=1,2$ in the progressive enlargement of $\mathbb{F}$ with $\tau_i$ is $\left(\alpha^i +\alpha^3\right) \left(\cdot\right)$.
\end{remark}

\textbf{Caveat.} In the rest of the paper, for easiness of notation, sometimes we write $\alpha^i_s$ instead of $\alpha^i\left(X_s\right)$ for $i=1,2,3$.

\subsection{Interpretation of Joint Distribution} \label{Interpretation.Distribution}
A way to interpret the distribution given in Theorem \ref{Thm.Survival Fn tau1, tau2} is the following: Suppose we have a two-component system where each component's life is represented by $\tau_1$ and $\tau_2$ respectively. Any component dies after receiving a shock. Shocks are governed by three independent Cox processes $\Lambda_1\left(t, \alpha^1(X_t)\right)$, $\Lambda_2\left(t, \alpha^2(X_t)\right)$ and $\Lambda_3\left(t, \alpha^3(X_t)\right)$, in other words, $\Lambda_i \left(t, \alpha^i\left(X_t\right) \right)$ is an inhomogeneous Poisson process with stochastic intensity $\alpha^i \left( X_t \right)$. For $j=1,2$, $\Lambda_j\left(t, \alpha^j(X_t)\right)$ represents the number of shocks through time that affect only component $j$, while $\Lambda_3\left(t, \alpha^3(X_t)\right)$ represents the number of shocks through time that affect both components.

Events in the process $\Lambda_1\left(t, \alpha^1(X_t)\right)$ are shocks to component 1, events in the process $\Lambda_2\left(t, \alpha^2(X_t)\right)$ are shocks to component 2 and events in the process $\Lambda_3\left(t, \alpha^3(X_t)\right)$ are shocks to both components.

Hence, we get,
\begin{multline}
	\Prob\left( \tau_1 > s, \tau_2>t \right) = \E \left[ \Prob\left( \tau_1 > s, \tau_2>t | (X_u)_{0\leq u \leq (s \vee t)} \right)   \right]  \\
	= \E \left[ \tilde{\Prob}_{(s,t)}\left[\Lambda_1\left(s, \alpha^1(X_s)\right)=0\right] \tilde{\Prob}_{(s,t)}\left[\Lambda_2\left(t, \alpha^2(X_t)\right)=0\right]\tilde{\Prob}_{(s,t)}\left[\Lambda_3\left(s\vee t, \alpha^3(X_{s\vee t})\right)=0\right] \right]    \\
	= \E\left[ \exp \left( - A^1_s - A^2_t - A^3_{s \vee t}\right) \right]
\end{multline}
Which coincides with Theorem \ref{Thm.Survival Fn tau1, tau2}. In the second line of the previous expression $\tilde{\Prob}_{(s,t)}(\cdot)$ means $\Prob\left(\cdot | (X_u)_{0\leq u \leq (s \vee t)} \right)$.

\subsection{Decomposition of the Joint Distribution} \label{Decomposition.Joint.Dist}
As found in Theorem \ref{Thm.Survival Fn tau1, tau2} and assuming that $X_t = t \in \mathbb{R}^+$, we have that the joint survival function is:
\begin{equation*} 
	\overline{F}_{(\tau_1, \tau_2)}(s,t) =  \exp\left[ -A^1_s -A^2_t - A^3_{s \vee t} \right] 
\end{equation*}
\begin{theorem} \label{Theorem.Joint.Distribution}
	Under the conditions of Theorem \ref{Thm.Survival Fn tau1, tau2} and assuming that $X_t = t \in \mathbb{R}$, the absolutely continuous and singular parts of $\overline{F}_{(\tau_1, \tau_2)}(s,t)$ are given by:
	\begin{equation}
		\overline{F}_{(\tau_1, \tau_2)}(s,t) = \beta \overline{F}_{\text{a}}(s,t) + (1-\beta) \overline{F}_{\text{sing}}(s,t)
	\end{equation}
	where
	\begin{align*}
		\beta  :& = \int\limits_0^\infty \left(\alpha^1_s+\alpha^2_s\right)  \exp\left[-A^1_s - A^2_s - A^3_s  \right] ds \\
		\overline{F}_{\text{a}}(s,t) & = \frac{1}{\beta} \left[e^{-A^1_s-A^2_t-A^3_{s \vee t} }- \int_{s \vee t }^\infty \alpha^3_x \exp \left(-\sum_{i=1}^3 A^i_x \right) dx \right] \\
		\overline{F}_{\text{sing}}(s,t) & = \frac{1}{1-\beta}\left[ \int_{s \vee t }^\infty \alpha^3_x \exp \left[-A^1_x-A^2_x-A^3_x \right] dx \right].
	\end{align*}
\end{theorem}

\begin{proof}
	Let $\Lambda_i$ be the waiting time to the first shock in the process $\Lambda_i(t, \alpha^i_t)$ defined in Section \ref{Interpretation.Distribution}. (Recall that we are assuming that $X_t = t$). By the assumptions made in Section \ref{Interpretation.Distribution}, $\Lambda_i$ are independent and the survival function of each one of them is given by:
	\begin{equation}
		\Prob\left(\Lambda_i > t \right) = \exp \left[-A^i_t\right] \text{ for all } t\geq 0.
	\end{equation}
	Hence, the density is equal to:
	\begin{equation} \label{Density.Lambda.i}
		f_{\Lambda_i}(s)  =\alpha^i_s\exp\left( -A^i_s \right) \text{ for all } s\geq 0.
	\end{equation}
	$\overline{F}_{(\tau_1, \tau_2)}(s,t)$ can be written as,
	\begin{equation}
		\overline{F}_{(\tau_1, \tau_2)}(s,t) = \Prob\left( \tau_1> s, \tau_2> t | B \right) \Prob\left( B \right) + \Prob\left( \tau_1> s, \tau_2> t | B^c \right) \Prob\left(B^c\right)
	\end{equation}
	where $B: = \left\{ \Lambda_3 > \min(\Lambda_1, \Lambda_2) \right\}$ and $B^c = \left\{ \Lambda_3 \leq \min(\Lambda_1, \Lambda_2) \right\}$ \par
	Now, using the density given in \eqref{Density.Lambda.i}:
	\begin{align}
		\Prob\left( B^c \right)  = & \Prob\left( \Lambda_3\leq \Lambda_1, \ \Lambda_3\leq \Lambda_2 \right) \nonumber \\
		= & \Prob\left( \Lambda_3 \leq \Lambda_1 \leq \Lambda_2 \right) + \Prob\left( \Lambda_3 \leq \Lambda_2 \leq \Lambda_1 \right) \nonumber \\
		= & \int\limits_0^\infty \int\limits_z^\infty \int\limits_x^\infty \alpha^1_x \alpha^2_y\alpha^3_z\exp\left[ -A^1_x -A^2_y -A^3_z \right] dy dx dz \nonumber \\
		& + \int\limits_0^\infty \int\limits_z^\infty \int\limits_y^\infty \alpha^1_x \alpha^2_y\alpha^3_z\exp\left[ -A^1_x -A^2_y -A^3_z \right] dx dy dz \nonumber \\
		= & \int_0^\infty \alpha^3_s \exp\left[ -A^1_s - A^2_s -A^3_s \right] ds
	\end{align}
	From this, $\Prob\left( B \right) = \int_0^\infty \left( \alpha^1_s + \alpha^2_s \right) \exp\left[ -A^1_s - A^2_s -A^3_s \right] ds$
	
	Since $\Prob\left(\min(\Lambda_1, \Lambda_2, \Lambda_3)>t\right) = \prod_{i=1}^3 \Prob\left(\Lambda_i>t \right) = \exp\left[-\sum_{i=1}^3 A^i_t \right]$, we get
	\begin{align}
		\Prob\left( \tau_1>s,\tau_2>t | B^c \right) & = \Prob\left[ \Lambda_1>s, \Lambda_2>t, \Lambda_3 > (s\vee t) | \Lambda_3\leq \min(\Lambda_1, \Lambda_2) \right] \nonumber \\
		&= \Prob\left[  \Lambda_3 > (s\vee t) | \Lambda_3\leq \min(\Lambda_1, \Lambda_2) \right] \nonumber \\
		& = \frac{\Prob\left[\left( s \vee t \right) < \Lambda_3 \leq \min(\Lambda_1, \Lambda_2)  \right]}{\Prob\left[ \Lambda_3 \leq \min(\Lambda_1, \Lambda_2) \right]} \nonumber \\
		& = \frac{\int_{s \vee t }^\infty \alpha^3_x \exp \left[-A^1_x-A^2_x-A^3_x \right] dx }{\int_{0}^\infty \alpha^3_x \exp \left[-A^1_x-A^2_x-A^3_x \right] dx}
	\end{align}
	With all these elements, $\Prob\left( \tau_1>s,\tau_2>t | B \right)$ is obtained by subtraction. 
	
	We can verify that $\Prob\left( \tau_1>s,\tau_2>t | B^c \right)$ is a singular distribution since its mixed second partial derivative is zero when $s\neq t$. Conversely, $\Prob\left( \tau_1>s,\tau_2>t | B \right)$ is absolutely continuous since its mixed second partial derivative is a density.
\end{proof}

\begin{remark}
	The value of $\Prob\left( B^c \right)$ corresponds to $\Prob(\tau_1 = \tau_2)$. This is, 
	\begin{equation} \label{Prob.Tau1.Equal.Tau2}
		\Prob(\tau_1 = \tau_2 ) = \int_0^\infty \alpha^3_s \exp\left[ -A^1_s - A^2_s -A^3_s \right] ds.
	\end{equation}
\end{remark}
\begin{corollary} \label{Prob.tau1.equal.tau2.random}
	Let $\alpha^i_u:= \alpha^i(X_u)$, where $\alpha^i(\cdot)$ is a positive continuous function and $X$ is an $\mathbb{R}^d$-valued stochastic process adapted to the filtration $\mathbb{F}$ and independent of $Z_1, Z_2 \text{ and } Z_3$. Then,
	\begin{equation*} 
		\Prob\left( \tau_1 = \tau_2 \right) = \E \left(\int_0^\infty \alpha^3_s \exp\left[ -A^{1}_s - A^2_s -A^3_s \right] ds\right).
	\end{equation*}
\end{corollary}
\begin{proof}
	By a similar calculation to get $\Prob(B^c)$ in the proof of Theorem \ref{Theorem.Joint.Distribution}, we get
	\begin{equation*}
		\Prob\left( \tau_1 = \tau_2  | (X_u)_{u \geq 0}\right) = \int_0^\infty \alpha^3_s \exp\left[ -A^{1}_s - A^2_s -A^3_s \right] ds.
	\end{equation*}
	The result follows by taking the expectation.
\end{proof}

\subsection{Estimating the Probability of Equality in Two Stopping Times} \label{Subsect.Singular.Part}
Now, we are interested in finding estimates for $\Prob(\tau_1 =\tau_2)$ (see equation (\ref{Prob.Tau1.Equal.Tau2}) in Section \ref{Decomposition.Joint.Dist}) under different assumptions for $\alpha^i(\cdot)$ (the intensity of $\eta_i$) or $A^i_s$ (the compensator of $\eta_i$). As in Theorem \ref{Thm.Survival Fn tau1, tau2}, we assume $\lim_{s\rightarrow\infty}A^i_s=\infty$ a.s. for every $i=1,2,3$.

\begin{example}[Constant intensity] \label{Example.Constant.Intensity}
	If $\alpha^i(X_s)=\alpha^i \in \mathbb{R}^+$ for all $s\geq 0$, it follows that:
	\begin{align*}
		\Prob(\tau_1=\tau_2) &= \int_0^\infty \alpha^3  \exp\left[-(\alpha^1+\alpha^2+\alpha^3)s  \right] ds  \\
		& = \frac{\alpha^3}{\alpha^1+\alpha^2+\alpha^3}
	\end{align*}
\end{example}

\begin{example}[Same intensity] \label{Proposition.Same.Intensity}
	If $\alpha^1(X_s) = \alpha^2(X_s) =\alpha^3(X_s)=:\alpha(X_s) $ for all $s\geq 0$, we have that $A^1_s=A^2_s=A^3_s=: A_s$ for all $s\geq 0$. Moreover, it is straighforward to get,
	\begin{equation*}
		\Prob\left( \tau_1 = \tau_2 \right) = \frac{1}{3}
	\end{equation*}
\end{example}

\begin{example}[Bounded intensity]
	Assume $L_i \beta(X_s) \leq \alpha^i(X_s) \leq U_i \beta(X_s) $ a.s. for all $s\geq 0 $, for $\beta(\cdot)$ positive, non-random, and integrable with $\int_0^\infty \beta(X_s)ds = \infty$ a.s., where $L_i, U_i$ are positive real random variables, independent of everything else, and such that $0< L_i < U_i <  \infty$. Then, we get
	\begin{equation*}
		\E \left[\frac{\ell_3}{u_1+u_2+u_3} \right] \leq \Prob(\tau_1 =\tau_2) \leq \E \left[\frac{u_3}{\ell_1 +\ell_2 + \ell_3}\right]
	\end{equation*}
\end{example}

\begin{example}[Intensity bounded by sum of intensities]
	If $L \left[ \alpha^1(X_s) +\alpha^2(X_s)\right] \leq \alpha^3(X_s) \leq U \left[ \alpha^1(X_s) +\alpha^2(X_s) \right]$ for a.s. all $s\geq 0$ where $L, U$ are positive real random variables, independent of everything else, and such that $0<L < U < L +1 <  \infty$. Then:
	\begin{equation*}
		\E \left[ \frac{L}{U+1} \right ] \leq \Prob(\tau_1 = \tau_2) \leq \E \left[ \frac{U}{L+1} \right]
	\end{equation*}
\end{example}
The proofs of these last two examples are straightforward conditioning arguments and, hence, left to the reader.

\subsection{Conditional Probabilities} \label{Subsect.Conditional.Probability}
Through this Section,  let $\tilde{\Prob}(\cdot):=\Prob\left(\cdot| \left( X_u \right)_{ u \geq 0 } \right)$.

\begin{proposition} \label{Proposition.Tau1.equal.Tau2.until.t}
	\begin{equation*}
		\Prob\left(\tau_1=\tau_2, \tau_1\leq t \right) = \E\left[ \int_0^t \alpha^3_s \exp\left[ -A^1_s - A^2_s -A^3_s \right] ds \right]
	\end{equation*}     
\end{proposition}

\begin{proof}
	By the definition of $\tau_i$, we have
	\begin{align*}
		\tilde{\Prob}(\tau_1=\tau_2,  \tau_1> t) = &  \tilde{\Prob}\left( \min(\eta_1, \eta_3) =\min(\eta_2, \eta_3), \min(\eta_1, \eta_3)> t\right)   \nonumber \\
		= &\tilde{\Prob}\left( \min(\eta_1, \eta_3) = \eta_3 =\min(\eta_2, \eta_3), \min(\eta_1, \eta_3)> t 
		\right) \nonumber \\
		= & \tilde{\Prob}\left( \eta_3 < \eta_1, \eta_3< \eta_2, \eta_1 > t, \eta_3 > t \right) \nonumber \\
		= & \tilde{\Prob}\left( t < \eta_3 < \eta_1 <\eta_2 \right) +\tilde{\Prob}\left( t < \eta_3 < \eta_2 <\eta_1 \right)
	\end{align*}
	Recall that, by definition, $\eta_i$ are independent with density under $\tilde{\Prob}$ equal to $\alpha^i_s\exp\left( -A^i_s \right)$
	\begin{align*}
		\tilde{\Prob}(\tau_1=\tau_2, \tau_1> t) = & \int\limits_t^\infty \int\limits_z^\infty \int\limits_x^\infty \alpha^1_x \alpha^2_y\alpha^3_z\exp\left[ -A^1_x -A^2_y -A^3_z \right] dy dx dz \nonumber \\
		& + \int\limits_t^\infty \int\limits_z^\infty \int\limits_y^\infty \alpha^1_x \alpha^2_y\alpha^3_z\exp\left[ -A^1_x -A^2_y -A^3_z \right] dx dy dz \nonumber \\
		= & \int_t^\infty \alpha^3_s \exp\left[ -A^1_s - A^2_s -A^3_s \right] ds
	\end{align*}
	This implies,
	\begin{equation*}
		\tilde{\Prob}(\tau_1=\tau_2, \tau_1\leq t)=\int_0^t \alpha^3_s \exp\left[ -A^1_s - A^2_s -A^3_s \right] ds
	\end{equation*}
	Take the expectation to get the result.
\end{proof}

\begin{remark}
	When taking the limit $t\rightarrow\infty$ in the previous proposition, we recover, as expected, the result of Corollary \ref{Prob.tau1.equal.tau2.random}.
\end{remark}

The following two propositions are straightforward applications of Proposition \ref{Proposition.Tau1.equal.Tau2.until.t}. Hence, the proofs are left to the reader.

\begin{proposition}
	\begin{equation*}
		\Prob\left( \tau_1 = \tau_2 | \tau_1 \leq t \right) = \E \left [\frac{\int_0^t \alpha^3_s \exp\left[ -A^1_s - A^2_s -A^3_s \right] ds}{1 - \exp[-A^1_t-A^3_t]}\right]
	\end{equation*}
\end{proposition}

\begin{proposition}
	\begin{equation*}
		\Prob\left( \tau_1 = \tau_2 | \tau_1 \leq t, \tau_2\leq t \right)  = \E \left[ \frac{\int_0^t \alpha^3_s \exp\left[ -A^1_s - A^2_s -A^3_s \right] ds}{1 - e^{-A^3_t}\left( e^{-A^2_t}+e^{-A^1_t}- e^{-A^1_t-A^2_t} \right)} \right]
	\end{equation*}
\end{proposition}

\subsection{Distance Between Stopping Times} \label{Subsect.Distance.Between.ST}
\begin{proposition} \label{Prop.Prob.Quadrant} \par
	For $s < t$,
	\begin{multline*}
		\Prob\left( s < \tau_1 \leq t, s \leq \tau_2 \leq t \right)  =  \E \left[ e^{-\left(A^1_s+A^2_s+A^3_s\right)} + e^{-\left(A^1_t+A^2_t+A^3_t\right)} \right. \\
		\left. - e^{-A^1_s-A^2_t-A^3_t} - e^{-A^1_t-A^2_s-A^3_t} \right]
	\end{multline*}
	In particular, if $s=r-\varepsilon$ and $t=r+\varepsilon$, we get
	\begin{multline*}
		\Prob\left( r-\varepsilon < \tau_1 \leq r+\varepsilon, r-\varepsilon \leq \tau_2 \leq r+\varepsilon \right)  =  \E \left[ e^{-\left(A^1_{r-\varepsilon}+A^2_{r-\varepsilon}+A^3_{r-\varepsilon}\right)} \right. \\ 
		\left. + e^{-\left(A^1_{r+\varepsilon}+A^2_{r+\varepsilon}+A^3_{r+\varepsilon}\right)} - e^{-A^1_{r-\varepsilon}-A^2_{r+\varepsilon}-A^3_{r+\varepsilon}} - e^{-A^1_{r+\varepsilon}-A^2_{r-\varepsilon}-A^3_{r+\varepsilon}} \right]
	\end{multline*}
\end{proposition}

\begin{proof}
	It follows by noticing that:
	\begin{multline*}
		\Prob\left( s < \tau_1 \leq t, s \leq \tau_2 \leq t \right)  = \Prob\left( \tau_1 > s, \tau_2> s \right) - \Prob\left( \tau_1 >s, \tau_2> t \right) \\
		- \Prob\left( \tau_1 > t, \tau_2> s \right) + \Prob\left( \tau_1 > t, \tau_2> t \right)
	\end{multline*}
	Then, conditioning on $\left( X_t\right)_{t\geq 0}$, use Theorem \ref{Thm.Survival Fn tau1, tau2}, and take the expectation.
\end{proof}

The next proposition shows that the probability of the two stopping times happening over the next time interval $(t,t+\varepsilon)$ with $\varepsilon\approx0$ is equal to the expectation of $\alpha^3\left(X_t\right)$, which is the common part of the intensities of $\tau_1$ and $\tau_2$ (recall Remark \ref{Intensity.Tau.i})

An anonymous referee has pointed out a relation of the following proposition to Aven's Lemma. This is, of course, correct, see Aven \cite{Aven}, Ethier and Kurtz \cite{Ethier.Kurtz}, and Zeng \cite{Zeng}.

\begin{proposition} \label{Prop.Limiting.Result}
	\begin{equation*}
		\lim_{\varepsilon\rightarrow0}\frac{\mathbb{P}\left(\tau_{1}\in\left(t,t+\varepsilon\right],\tau_{2}\in\left(t,t+\varepsilon\right]|\tau_{1}>t,\tau_{2}>t\right)}{\varepsilon}=\mathbb{E}\left(\alpha^3(X_{t})\right)
	\end{equation*}
\end{proposition}

\begin{proof}
	Let $\tilde{\Prob}(\cdot):=\Prob\left(\cdot| \left( X_u \right)_{ u \geq 0 } \right)$ and note that:
	\begin{equation*}
		\tilde{\mathbb{P}} \left(\tau_{1}\in\left(t,t+\varepsilon\right],\tau_{2}\in\left(t,t+\varepsilon\right]|\tau_{1}>t,\tau_{2}>t\right) =\frac{\tilde{\mathbb{P}}\left(\tau_{1}\in\left(t,t+\varepsilon\right],\tau_{2}\in\left(t,t+\varepsilon\right]\right)}{\tilde{\mathbb{P}}\left(\tau_{1}>t,\tau_{2}>t\right)}
	\end{equation*}
	Using Proposition \ref{Prop.Prob.Quadrant} with $t$ and $t+\varepsilon$, dividing by $\varepsilon$, and using L'H\^opital's rule to get the limit: 	
	\begin{equation*}
		\lim_{\varepsilon\rightarrow0} \frac{\tilde{\mathbb{P}}\left(\tau_{1}\in\left(t,t+\varepsilon\right],\tau_{2}\in\left(t,t+\varepsilon\right]|\tau_{1}>t,\tau_{2}>t \right)}{\varepsilon} = \alpha^3(X_{t}).
	\end{equation*}
	Finally, given $\tilde{\mathbb{P}}\left(\tau_{1}\in\left(t,t+\varepsilon\right],\tau_{2}\in\left(t,t+\varepsilon\right]|\tau_{1}>t,\tau_{2}>t\right)$
	is bounded by 1 and using a conditioning argument, we can conclude that:
	\begin{equation*}
		\lim_{\varepsilon\rightarrow0}\frac{\mathbb{P}\left(\tau_{1}\in\left(t,t+\varepsilon\right],\tau_{2}\in\left(t,t+\varepsilon\right]|\tau_{1}>t,\tau_{2}>t\right)}{\varepsilon}=\mathbb{E}\left(\alpha^3(X_{t})\right).
	\end{equation*}
\end{proof}

Proposition \ref{Prop.Limiting.Result} motivates the following one:

\begin{proposition} \par
	If $s < t$ 
	\begin{equation*}
		\lim_{\varepsilon\rightarrow0}\frac{\mathbb{P}\left(\tau_{1}\leq s+\varepsilon,\tau_{2}\leq t+\varepsilon|\tau_{1}>s,\tau_{2}>t\right)}{\varepsilon^2}=\mathbb{E}\left[\alpha^1(X_s)\left( \alpha^2(X_t)+\alpha^3(X_t) \right)\right]
	\end{equation*}
\end{proposition}

\begin{proof}
	As $s<t$, we can always find a sufficiently small $\varepsilon$ such that $s+\varepsilon<t$. Hence, without loss of generality, we assume $s<s+\varepsilon<t< t+\varepsilon$ and we proceed as in the proof of Proposition \ref{Prop.Limiting.Result}, but differently from there, we use  L'H\^opital's rule twice to get the desired limit.
\end{proof}

The next couple of propositions provide a measure, in two different metrics, of how close the two stopping are from each other.

\begin{proposition} (Distance in probability)
	\begin{equation*}
		\Prob(|\tau_1-\tau_2|\leq \varepsilon) =  1 - \E\left[  \int\limits_0^\infty \alpha^1_x e^{-A^1_x-\left(A^2_{x+\varepsilon}+A^3_{x+\varepsilon}\right)} \, dx  + \int\limits_0^\infty \alpha^2_x e^{-A^2_x-\left(A^1_{x+\varepsilon}+A^3_{x+\varepsilon}\right)} \, dx  \right]   \nonumber 
	\end{equation*}
\end{proposition}

\begin{proof}
	Let $\tau_{(1)} := \min \left(\tau_1, \tau_2 \right)$ and $\tau_{(2)} := \max \left(\tau_1, \tau_2 \right)$. Similarly, $\eta_{(1)} := \min \left(\eta_1, \eta_2, \eta_3 \right)$, $\eta_{(3)} := \max \left(\eta_1, \eta_2, \eta_3 \right)$, and $\eta_{(2)}$ be the second largest from $\left( \eta_1, \eta_2, \eta_3 \right)$. Also, let $\tilde{\Prob}(\cdot):=\Prob\left(\cdot| \left( X_u \right)_{ u \geq 0 } \right)$. Then,
	\begin{align*}
		\tilde{\mathbb{P}}  \left( \tau_{(2)}  > \tau_{(1)} + \varepsilon  \right)  = & \tilde{\mathbb{P}} \left( \eta_{(3)} \geq \eta_{(2)} > \eta_{(1)} +\varepsilon, \eta_3 \geq \eta_{(2)}  \right) \\
		= & \int\limits_0^\infty f_1(x_1) \int\limits_{x_1+\varepsilon}^\infty f_2(x_2) \int\limits_{x_2}^{\infty} f_3(x_3) \, dx_3 \,  dx_2 \, dx_1 \\
		&+ \int\limits_0^\infty f_2(x_1) \int\limits_{x_1+\varepsilon}^\infty f_1(x_2) \int\limits_{x_2}^{\infty} f_3(x_3) \, dx_3 \,  dx_2 \, dx_1\\
		& + \int\limits_0^\infty f_1(x_1) \int\limits_{x_1+\varepsilon}^\infty f_3(x_2) \int\limits_{x_2}^{\infty} f_2(x_3) \, dx_3 \,  dx_2 \, dx_1 \\
		&+ \int\limits_0^\infty f_2(x_1) \int\limits_{x_1+\varepsilon}^\infty f_3(x_2) \int\limits_{x_2}^{\infty} f_1(x_3) \, dx_3 \,  dx_2 \, dx_1 \\
		\intertext{ where $f_j$ stands for the density of $\eta_j$ under $\tilde{\mathbb{P}}$, i.e., $f_j(y)=\alpha^j_ye^{-A^j_y}:=\alpha^j(X_y)e^{-A^j_y}$} 
		= & \int\limits_0^\infty f_1(x_1) \int\limits_{x_1+\varepsilon}^\infty f_2(x_2) e^{-A^3_{x_2}} \,  dx_2 \, dx_1 + \int\limits_0^\infty f_2(x_1) \int\limits_{x_1+\varepsilon}^\infty f_1(x_2) e^{-A^3_{x_2}} \,  dx_2 \, dx_1 \\
		& + \int\limits_0^\infty f_1(x_1) \int\limits_{x_1+\varepsilon}^\infty f_3(x_2) e^{-A^2_{x_2}} \,  dx_2 \, dx_1 + \int\limits_0^\infty f_2(x_1) \int\limits_{x_1+\varepsilon}^\infty f_3(x_2) e^{-A^1_{x_2}} \,  dx_2 \, dx_1 \\
		= & \int\limits_0^\infty f_1(x_1) \int\limits_{x_1+\varepsilon}^\infty \left[ \alpha^2_{x_2} + \alpha^3_{x_2} \right] e^{-A^2_{x_2}-A^3_{x_2}} \,  dx_2 \, dx_1 \\
		& + \int\limits_0^\infty f_2(x_1) \int\limits_{x_1+\varepsilon}^\infty \left[ \alpha^1_{x_2} + \alpha^3_{x_2} \right] e^{-A^1_{x_2}-A^3_{x_2}} \,  dx_2 \, dx_1 \\
		= & \int\limits_0^\infty \alpha^1_{x_1} e^{-A^1_{x_1}-\left(A^2_{x_1 +\varepsilon}+A^3_{x_1 +\varepsilon}\right)} \, dx_1 + \int\limits_0^\infty \alpha^2_{x_1} e^{-A^2_{x_1}-\left(A^1_{x_1 +\varepsilon}+A^3_{x_1 +\varepsilon}\right)} \, dx_1
	\end{align*}
	Use this expression, along with the following equality, and then take the expectation to get the desired result
	\begin{align*}
		\tilde{\mathbb{P}} \left( |\tau_1 - \tau_2| \leq \varepsilon \right) = 1 - \tilde{\mathbb{P}}\left(\tau_{(2)} - \tau_{(1)} > \varepsilon \right)
	\end{align*}
\end{proof}

\begin{proposition}[$L^2$ distance] \label{Prop.L2.Distance} \par
	If $\E\left( \tau_i^2 \right)<\infty$ and $\lim_{x\rightarrow\infty} x^2 e^{-A^i_x}=0$ a.s. for $i=1,2$ (in other words, $e^{-A^i_x}$ goes faster to 0 than $x^2$ goes to infinity when $x\rightarrow\infty$), we have
	\begin{multline*}
		\E\left[\left( \tau_1 -\tau_2 \right)^2 \right] =  2 \E \left[\int_0^\infty xe^{-A^1_x-A^3_x}dx + \int_0^\infty xe^{-A^2_x-A^3_x}dx \right.  \\
		\left. - \int_0^\infty \int_0^y  e^{-A^1_x-A^2_y-A^3_y} dx dy - \int_0^\infty \int_y^\infty  e^{-A^1_x-A^2_y-A^3_x} dx dy \right] 
	\end{multline*}
\end{proposition}
\begin{proof}
	Let $\tilde{\mathbb{E}}(\cdot) := \E(\cdot | (X_{u})_{u\geq 0})$. $F_{\tau_i}$ and $\overline{F}_{\tau_i}$ stand for the cumulative and the survival distribution functions of $\tau_i$ given $(X_{u})_{u\geq 0}$. Similarly, $F(x,y)$ and $\overline{F}(x,y)$ stand for the cumulative and the survival joint distribution functions of $\tau_1, \tau_2$ given $(X_{u})_{u\geq 0}$. \par
	We expand the square and handle each term separately. For $\tilde{\mathbb{E}}\left( \tau_i^2 \right)<\infty$, we use integration by parts in the following way:
	\begin{align*}
		\tilde{\mathbb{E}}\left( \tau_i^2 \right) &= \int_0^\infty x^2 dF_{\tau_i}(x) \nonumber \\
		&= 2\int_0^\infty x \overline{F}_{\tau_i}(x) dx \nonumber \\
		&= 2\int_0^\infty xe^{-A^i_x-A^3_x}dx
	\end{align*}
	To find $\tilde{\mathbb{E}}(\tau_1 \tau_2) $, we exploit the result of Young \cite{Young} on integration by parts in two or more dimensions. If $G(0, y)\equiv 0 \equiv G(y,0)$ and $G$ is of bounded variation on finite intervals, then:
	\begin{equation}
		\int_0^\infty \int_0^\infty G(x,y) F(dx,dy) = \int_0^\infty \int_0^\infty \overline{F}(x,y) G(dx,dy) 
	\end{equation}
	This equality implies that, for $i,j>0$:
	\begin{equation}
		\int_0^\infty \int_0^\infty  x^i y^j F(dx,dy) = \int_0^\infty \int_0^\infty i j x^{i-1} y^{j-1} \overline{F}(x,y)dx dy
	\end{equation}
	Hence,
	\begin{align} \label{Expectation.Tau1.Tau2}
		\tilde{\mathbb{E}}\left( \tau_1 \tau_2 \right) & = \int_0^\infty \int_0^\infty x y F_{\left(\tau_1, \tau_2\right)}(dx, dy) \nonumber \\
		& = \int_0^\infty \int_0^\infty  e^{-A^1_x-A^2_y-A^3_{x \vee y}} dx dy \nonumber \\
		& = \int_0^\infty \int_0^y  e^{-A^1_x-A^2_y-A^3_y} dx dy + \int_0^\infty \int_y^\infty  e^{-A^1_x-A^2_y-A^3_x} dx dy
	\end{align}
	The result follows by taking the expectation.
\end{proof}

\begin{remark}
	In the case of independence of $\tau_1$ and $\tau_2$, as $A^3_s=0$ for all $s\geq 0$, we get that:
	\begin{equation}
		\E\left[\left( \tau_1 -\tau_2 \right)^2 \right] = 2 \E \left[ \int_0^\infty xe^{-A^1_x}dx +  \int_0^\infty xe^{-A^2_x}dx  - \int_0^\infty e^{-A^1_x}dx \int_0^\infty e^{-A^2_x} dx \right]  
	\end{equation}
\end{remark}

\section{GENERALIZATIONS} \label{Sect.Generalization}
\subsection{Generalization to K Stopping Times.}  \label{Sect.Generalization.K.ST}
Given the interpretation introduced in Section (\ref{Interpretation.Distribution}), there is a natural way to extend our model to more than two stopping times. We explicitly motivate and present the case of 3 stopping times. Suppose we have a three-component system where each component's life is represented by $\tau_1$, $\tau_2$, and $\tau_3$ respectively. Any component dies after receiving a shock, which are governed by 6 Cox processes, which are independent given $(X_t)_{t \geq 0} $:
\begin{multline*}
	\Lambda_1\left(t, \alpha^1(X_t)\right), \Lambda_2\left(t, \alpha^2(X_t)\right), \Lambda_3\left(t, \alpha^3(X_t)\right), \Lambda_{(1,2)}\left(t, \alpha^{(1,2)}(X_t)\right), \Lambda_{(2,3)}\left(t, \alpha^{(2,3)}(X_t)\right), \\
	\Lambda_{(1,3)}\left(t, \alpha^{(1,3)}(X_t)\right) \text{ and } \Lambda_{(1,2,3)}\left(t, \alpha^{(1,2,3)}(X_t)\right)
\end{multline*}
Events in the process $\Lambda_i\left(t, \alpha^i(X_t)\right)$ are shocks to only component $i$ (for $i=1,2,3$), events in the process $\Lambda_{(i,j)}\left(t, \alpha^{(i,j)}(X_t)\right)$ are shocks to component $i$ and $j$ (for $i,j=1,2,3$ and $i\neq j$), and events in the process $\Lambda_{(1,2,3)}\left(t, \alpha^{(1,2,3)}(X_t)\right)$ are shocks to the 3 components.

In this way, considering $\tilde{\Prob}(\cdot):=\Prob\left(\cdot | (X_u)_{u\geq 0}  \right)$ and $A^j_s = \int_0^s \alpha^j(X_u)du$ we have:
\begin{multline} \label{Generalization.1}
	\Prob(\tau_1> s_1,   \tau_2> s_2, \tau_3> s_3) = \E \Bigl( \tilde{\mathbb{P}}\left[ \Lambda_1\left(s_1, \alpha^1(X_{s_1})\right) = 0 \right] \tilde{\mathbb{P}}\left[ \Lambda_2\left(s_2, \alpha^2(X_{s_2})\right) = 0 \right] \\
	\tilde{\mathbb{P}}\left[ \Lambda_3\left(s_3, \alpha^3(X_{s_3})\right) = 0 \right]  \tilde{\mathbb{P}}\left[ \Lambda_{(1,2)}\left(s_1\vee s _2, \alpha^{(1,2)}(X_{s_1\vee s_2})\right) = 0 \right] \tilde{\mathbb{P}}\left[ \Lambda_{(1,3)}\left(s_1\vee s _3, \alpha^{(1,3)}(X_{s_1 \vee s_3})\right) = 0 \right] \\
	\tilde{\mathbb{P}}\left[ \Lambda_{(2,3)}\left(s_2\vee s _3, \alpha^{(2,3)}(X_{s_2 \vee s_3})\right) = 0 \right]  \tilde{\mathbb{P}}\left[ \Lambda_{(1,2,3)}\left(s_1\vee s _2\vee s_3, \alpha^{(1,2,3)}(X_{s_1\vee s_2 \vee s_3})\right) = 0 \right] \Bigr) \\
	= \E \left( \exp\left[ -A^1_{s_1} - A^2_{s_2} - A^3_{s_3}  -  A^{(1,2)}_{s_1 \vee s_2}-  A^{(1,3)}_{s_1 \vee s_3} -  A^{(2,3)}_{s_2 \vee s_3} -  A^{(1,2, 3)}_{s_1 \vee s_2 \vee s_3} \right] \right)
\end{multline}
By using a similar technique we can generalize to any number $K$ of stopping times and in Section \ref{Subsect.Application.Epidemiology}, we will present an application of this natural extension of our model. However, as the number of stopping times $K$ increases, handling the expression presented in equation \eqref{Generalization.1} becomes cumbersome. In Jarrow et al. \cite{Jarrow-Protter-Quintos}, we propose and study in detail a slightly different approach that makes easier to generalize to more than two stopping times.

\subsection{A More General Distribution} \label{Sect.A.More.General.Distribution}
In this Section, we generalize the result from Section \ref{Sect.Survival.Function} to get a distribution that allows $\tau_1$ and $\tau_2$ to have a negative covariance  \footnote{We thank Philip Ernst and Guodong Pang for suggesting this idea in an earlier version of this paper.}. Recalling Equation \eqref{Expectation.Tau1.Tau2} and that  $ \tilde{\E}  \left(\tau_i\right) = \int_0^\infty e^{-A^i_s-A^3_s}ds$, one can see that the covariance of $\tau_1$ and $\tau_2$ in the previous model is:
\begin{align}
	\C \left(\tau_1, \tau_2 \right) & =  \E \left[\int_0^\infty \int_0^\infty \left(  e^{-A^1_x-A^2_y-A^3_{x \vee y}} - e^{-A^1_x -A^2_y -A^3_x - A^3_y} \right) dx dy \right] \nonumber \\
	& = \E \left[ \int_0^\infty \int_0^\infty e^{-A^1_x-A^2_y}\left(e^{-A^3_{x \vee y}} - e^{-A^3_x -A^3_y}\right) dx dy \right]
\end{align}
As $A^3_{x\vee y} < A^3_x + A^3_y$ for all $x,y\geq 0$, it is clear that the covariance is always nonnegative.

To get the property of a negative covariance, we stop assuming that the underlying exponential random variables (i.e., $Z_1, Z_2, Z_3$) are independent (recall equation \eqref{Initial.Definitions} where we use these random variables to define $\eta_i$ and $A^i_s$). Now, we assume that $(Z_1, Z_2)$ follow the Gumbel Bivariate Distribution (see Gumbel \cite{Gumbel}). This is, the joint survival function of $(Z_1, Z_2)$ is:
\begin{equation}
	\Prob(Z_1>s, Z_2>t)=e^{-s-t-\delta s t} \text{ for } 0 \leq \delta \leq 1
\end{equation}
While $Z_3\sim \text{Exp}(1)$ and it is independent of $(Z_1, Z_2)$. The rest of the definitions given in \eqref{Initial.Definitions} remain the same. Then, it is easy to check that:
\begin{equation} \label{Gumbel.Survival.Eta.1.2}
	\Prob\left( \eta_1>s, \eta_2>t \right) = \E \left[ \exp\left(- A^1_s -A^2_t - \delta A^1_s A^2_t \right) \right]
\end{equation}
\begin{equation} \label{Gumbel.Survival.Eta.3}
	\Prob\left( \eta_3>s \right) =\E \left[  \exp \left( -A^3_s \right) \right]
\end{equation}
Define $\tau_1$ and $\tau_2$ as in equation \eqref{Main.Model}. Then, we can get the join distribution of $(\tau_1, \tau_2)$:
\begin{theorem} [Survival function] \label{Gumbel.Thm.Survival Fn tau1, tau2}
	Suppose $\alpha^i: \mathbb{R}^d \mapsto [0, \infty)$ for $i=1,2,3$ are non-random positive continuous functions, which implies that $A^i_s$ are continuous and strictly increasing. Assume further that $\lim_{s\rightarrow\infty}A^i_s=\infty$ a.s. Then,
	\begin{equation*} 
		\overline{F}_{(\tau_1, \tau_2)}(s,t):= \Prob\left( \tau_1 > s, \tau_2 >t\right) = \E \left[\exp\left[- A^1_s - A^2_t - \delta A^1_s A^2_t - A^3_{s \vee t} \right] \right] .
	\end{equation*}
\end{theorem}

\begin{proof}
	Let $\tilde{\Prob}_{(s,t)}(\cdot):= \Prob\left(\cdot |(X_u)_{0\leq u \leq (s \vee t)} \right)$. By definition of $\tau_1$ and $\tau_2$, we have,
	\begin{align*} 
		\tilde{\Prob}_{(s,t)}\left( \tau_1 > s, \tau_2 >t\right) & =  \tilde{\Prob}_{(s,t)} \left( \min(\eta_1, \eta_3) > s, \min(\eta_2, \eta_3) >t \right) \nonumber \\
		& = \tilde{\Prob}_{(s,t)} \left( \eta_1>s,  \eta_2 > t \right)  \tilde{\Prob}_{(s,t)} \left( \eta_3>s, \eta_3>t \right) \nonumber \\
		& = \tilde{\Prob}_{(s,t)} \left( \eta_1>s,  \eta_2 > t \right) \tilde{\Prob}_{(s,t)} \left( \eta_3>\max(s,t) \right) \nonumber \\
		& = \exp\left[ -A^1_s -A^2_t -\delta A^1_s A^2_t -A^3_{s \vee t}\right] \nonumber \\
		& = \begin{cases}
			\exp\left[ -A^1_s -A^2_t -\delta A^1_s A^2_t - A^3_s  \right] & s>t \\
			\exp\left[ -A^1_s -A^2_t -\delta A^1_s A^2_t - A^3_t  \right] & s \leq t.
		\end{cases}
	\end{align*}
	The result follows by taking the expectation.
\end{proof}

\begin{remark}
	From Theorem \ref{Gumbel.Thm.Survival Fn tau1, tau2}, it is clear that $\tau_1$ is not independent of $\tau_2$. 
\end{remark}

\begin{remark}
	The marginal distribution of $\tau_i$ is given by:
	\begin{equation} \label{Gumbel.Marginal.Survival.Function}
		\Prob(\tau_i>s) = \Prob(\min(\eta_i, \eta_3) > s) = \E \left[\exp\left(-\int_0^s (\alpha^i_r+\alpha^3_r)dr\right) \right]
	\end{equation}
	which, as in Section \ref{Sect.Survival.Function}, coincides with the marginal distribution in the Cox construction, see Lando \cite{Lando.Cox} \cite{Lando.Credit.Risk.Book}.
\end{remark}

\begin{proposition}[Probability of the two stopping times being equal] \par  
	\begin{equation}  \label{Gumbel.Proof.Tau1.Equal.Tau2}
		\Prob \left( \tau_1 = \tau_2 \right) = \E \left[ \int_0^\infty \alpha^3_t \exp\left[ -(A^1_t+A^2_t+A^3_t)  - \delta A^1_tA^2_t \right] dt \right]
	\end{equation}
	which is greater than $0$ as long as $\alpha^3_t$ is not identically equal to $0$
\end{proposition}

The proof of this Proposition follows by a similar argument as the one used in the proof of Theorem \ref{Theorem.Joint.Distribution} to get $\Prob\left(B^C\right)$, so we leave it to the reader.

This is the general set-up. For the rest of this Section, we assume that for $i=1,2,3$, we have $\alpha^i(X_t) \equiv \lambda_i \in \mathbb{R}^+$ for all $t\geq 0$ to get tractable computations. Under this assumption, note the following:
\begin{enumerate}
	\item $\Prob \left(\eta_1>s, \eta_2>t\right) = \exp \left[-\lambda_1 s - \lambda_2 t - \delta \lambda_1 \lambda_2 s t \right]$
	\item $\Prob\left( \tau_1>s, \tau_2> t \right) = \exp\left[ -\lambda_1 s - \lambda_2 t - \delta \lambda_1 \lambda_2 s t - \lambda_3 (s\vee t) \right]$. That is $(\tau_1,\tau_2)$ follow a bivariate exponential (BVE) that is a combination of the Marshall-Olkin BVE (see Marshall and Olkin \cite{MarshallOlkin}) and the Gumbel BVE (see \cite{Gumbel}). Actually, if $\delta=0$, we recover the BVE of Marshall and Olkin and if $\lambda_3=0$ we recover the BVE of Gumbel.
	\item $\Prob(\tau_1>t) = \exp\left(-(\lambda_1+\lambda_3)t\right) $ and $\Prob(\tau_2>t) = \exp\left(-(\lambda_2 +\lambda_3)t\right) $, i.e., marginally $\tau_i\sim\text{Exp}(\lambda_i+\lambda_3)$.
\end{enumerate}

Before showing that one can get a negative covariance between the stopping times, we show in the next proposition that, by allowing the possibility of a negative covariance, the probability of the two stopping times being equal is smaller.

\begin{proposition} \par
	Fix some values of $\lambda_1, \lambda_2, \lambda_3$. If $\delta\neq 0$, the probability of $\tau_1$ being equal to $\tau_2$ is smaller than the probability of $\tau_1$ being equal to $\tau_2$ for the case of $\delta=0$
\end{proposition}

\begin{proof}
	Let $\Prob_{\delta\neq 0}(A)$ and $\Prob_{\delta = 0}(A)$ stand for the probability of an event $A$ under the law of $(\tau_1, \tau_2)$ when $\delta\neq 0$ and when $\delta = 0$ respectively
	
	Recall from Example \ref{Example.Constant.Intensity} that:
	\begin{equation}
		\Prob_{\delta=0}\left(\tau_1 = \tau_2 \right) = \frac{\lambda_3}{\lambda_1+\lambda_2+\lambda_3}
	\end{equation}
	Also, setting $\alpha^i_t=\lambda_i$ in \eqref{Gumbel.Proof.Tau1.Equal.Tau2}, we get:
	\begin{align}
		\Prob_{\delta\neq 0}\left( \tau_1 =\tau_2 \right) &  = \int_0^\infty \lambda_3 \exp\left[ -(\lambda_1+ \lambda_2 + \lambda_3) t - \delta \lambda_1 \lambda_2 t^2 \right] dt  \nonumber \\
		& = \lambda_3 \exp\left[\frac{\left(\lambda_1+\lambda_2+\lambda_3\right)^2}{4 \delta \lambda_1 \lambda_2} \right] \int_0^\infty \exp \left[- \left(\sqrt{\delta \lambda_1 \lambda_2} t + \frac{\lambda_1 + \lambda_2 + \lambda_3}{2 \sqrt{\delta \lambda_1 \lambda_2}}\right)^2 \right] dt \nonumber \\
		& = \frac{\lambda_3}{\sqrt{\delta \lambda_1 \lambda_2}} \exp\left[\frac{\left(\lambda_1+\lambda_2+\lambda_3\right)^2}{4 \delta \lambda_1 \lambda_2} \right] \int_{\frac{\lambda_1 + \lambda_2 + \lambda_3}{2 \sqrt{\delta \lambda_1 \lambda_2}}}^\infty e^{-u^2} du
	\end{align}
	where the second equality follows by completing the square and it is only valid if $\delta \neq 0$. The third equality is just a change of variable.
	
	Then, using the well-known bound $\int_x^\infty e^{-u^2} du \leq \frac{e^{-x^2}}{2 x}  \text{ for all } x>0$, we have that:
	\begin{align*}
		\frac{\lambda_3}{\sqrt{\delta \lambda_1 \lambda_2}} & \exp\left[\frac{\left(\lambda_1+\lambda_2+\lambda_3\right)^2}{4 \delta \lambda_1 \lambda_2} \right] \int_{\frac{\lambda_1 + \lambda_2 + \lambda_3}{2 \sqrt{\delta \lambda_1 \lambda_2}}}^\infty e^{-u^2} du \\ 
		& \leq \frac{\lambda_3}{\sqrt{\delta \lambda_1 \lambda_2}} \exp\left[\frac{\left(\lambda_1+\lambda_2+\lambda_3\right)^2}{4 \delta \lambda_1 \lambda_2} \right] \exp\left[ -\frac{\left(\lambda_1+\lambda_2+\lambda_3\right)^2}{4 \delta \lambda_1 \lambda_2} \right] \frac{2 \sqrt{ \delta \lambda_1 \lambda_2}}{2\left(\lambda_1 + \lambda_2 + \lambda_3\right)} \\
		& = \frac{\lambda_3}{\lambda_1 + \lambda_2 + \lambda_3}
	\end{align*}
	which shows, as desired:
	\begin{equation}
		\Prob_{\delta\neq 0} \left(\tau_1  = \tau_2 \right) \leq \Prob_{\delta = 0} \left(\tau_1  = \tau_2 \right)
	\end{equation}
\end{proof}

\begin{proposition}[Negative Covariance] \label{Proposition.Neg.Cov} \par
	Suppose $\lambda_1 = c_1 \lambda_3$, $\lambda_2 = c_2 \lambda_3$ and $\delta=1$. If $5 c_1 c_2 \geq 4 \left(c_1 +c_2 +1\right)$, then $\text{Cov}\left(\tau_1, \tau_2\right) <0$.
\end{proposition}

\begin{proof}
	Using the result of Young \cite{Young}, as in the proof of Proposition \ref{Prop.L2.Distance} as well as Theorem \ref{Gumbel.Thm.Survival Fn tau1, tau2}, we have that:
	\begin{align} \label{Gumbel.Expectation.Tau1.Tau2.General}
		\E \left( \tau_1 \tau_2\right) = & \int_0^\infty \int_0^\infty \exp\left[ -\lambda_1 x - \lambda_2 y-
		\lambda_3 (x\vee y) - \lambda_1 \lambda_2 x y \right] dx dy \nonumber \\
		= & \int_0^\infty \int_0^y \exp\left[ -\lambda_1 x- \lambda_2  y -  \lambda_1 \lambda_2 x y - 
		\lambda_3 y  \right] dx dy \nonumber \\
		& + \int_0^\infty \int_y^\infty \exp\left[ -\lambda_1 x- \lambda_2  y -  \lambda_1 \lambda_2 x y - 
		\lambda_3 x  \right] dx dy \nonumber \\
		= & \int_0^\infty \int_x^\infty \exp\left[ -\lambda_1 x- \lambda_2  y -  \lambda_1 \lambda_2 x y - 
		\lambda_3 y  \right] dy dx \nonumber \\
		& + \int_0^\infty \int_y^\infty \exp\left[ -\lambda_1 x- \lambda_2  y -  \lambda_1 \lambda_2 x y - 
		\lambda_3 x  \right] dx dy \nonumber \\
		= & \int_0^\infty \frac{1}{\lambda_2 + \lambda_3 +  \lambda_1 \lambda_2 x } \exp \left[ -
		(\lambda_1+\lambda_2 + \lambda_3) x -  \lambda_1 \lambda_2 x^2 \right)]dx \nonumber \\
		& + \int_0^\infty \frac{1}{\lambda_1 + \lambda_3 +  \lambda_1 \lambda_2 y } \exp \left[ -
		(\lambda_1+\lambda_2 + \lambda_3) y -  \lambda_1 \lambda_2 y^2 \right)]dy 
	\end{align}
	We reduce the previous equation even further by completing the square, for instance, take the first
	summand:
	\begin{multline} \label{Proof.Neg.Cov.Integral1}
		\int_0^\infty \frac{1}{\lambda_2 + \lambda_3 + \lambda_1 \lambda_2 x } \exp \left[ -
		(\lambda_1+\lambda_2 + \lambda_3) x -  \lambda_1 \lambda_2 x^2 \right)]dx  \\
		= \exp \left[ \frac{\left( \lambda_1 + \lambda_2 + \lambda_3 \right)^2 }{ 4 \lambda_1 \lambda_2} \right] 
		\int_0^\infty  \frac{1}{\lambda_2 + \lambda_3 + \lambda_1 \lambda_2 x } \exp \left[  -  
		\left( \sqrt{\lambda_1 \lambda_2} x  + \frac{1}{2} \left( \frac{\lambda_1+ \lambda_2 +\lambda_3}
		{\sqrt{\lambda_1 \lambda_2}} \right) \right)^2 \right] dx \\
		= \frac{2}{\sqrt{\lambda_1 \lambda_2}} \exp \left[ \frac{\left( \lambda_1 + \lambda_2 + \lambda_3 
			\right)^2 }{ 4 \lambda_1 \lambda_2} \right] 
		\int\limits_{\frac{1}{2} \left( \frac{\lambda_1+ \lambda_2 +\lambda_3}
			{\sqrt{\lambda_1 \lambda_2}} \right)  }^\infty  \frac{1}{ 2 \sqrt{\lambda_1 \lambda_2} u + \lambda_2 + 
			\lambda_3 - \lambda_1} e^{-u^2} du 
	\end{multline}
	Similarly, the second summand in \eqref{Gumbel.Expectation.Tau1.Tau2.General} reduces to:
	\begin{multline} \label{Proof.Neg.Cov.Integral2}
		\int_0^\infty \frac{1}{\lambda_1 + \lambda_3 + \lambda_1 \lambda_2 y } \exp \left[ -
		(\lambda_1+\lambda_2 + \lambda_3) y -  \lambda_1 \lambda_2 y^2 \right)]dy \\
		= \frac{2}{\sqrt{\lambda_1 \lambda_2}} \exp \left[ \frac{\left( \lambda_1 + \lambda_2 + \lambda_3 
			\right)^2 }{ 4 \lambda_1 \lambda_2} \right] 
		\int\limits_{\frac{1}{2} \left( \frac{\lambda_1+ \lambda_2 +\lambda_3}
			{\sqrt{\lambda_1 \lambda_2}} \right)  }^\infty  \frac{1}{ 2 \sqrt{\lambda_1 \lambda_2} u + \lambda_1 + \lambda_3 - \lambda_2} e^{-u^2} du
	\end{multline}
	Using the fact that $\E \left(\tau_i\right)=\frac{1}{\lambda_i + \lambda_3}$ and equations  \eqref{Proof.Neg.Cov.Integral1} and \eqref{Proof.Neg.Cov.Integral2}, we get that:
	\begin{multline}
		\text{Cov}\left(\tau_1, \tau_2\right) =  \frac{2}{\sqrt{\lambda_1 \lambda_2}} \exp \left[ \frac{\left( \lambda_1 + \lambda_2 + \lambda_3 \right)^2 }{ 4 \lambda_1 \lambda_2} \right] 
		\int\limits_{\frac{1}{2} \left( \frac{\lambda_1+ \lambda_2 +\lambda_3}{\sqrt{\lambda_1 \lambda_2}} \right)  }^\infty  \left[  \frac{1}{ 2 \sqrt{\lambda_1 \lambda_2} u + \lambda_2 + \lambda_3 - \lambda_1}  \right. \\
		\left. + \frac{1}{ 2 \sqrt{\lambda_1 \lambda_2} u + \lambda_1 + \lambda_3 - \lambda_2}  \right]e^{-u^2} du - \left(\frac{1}{\lambda_1 + \lambda_3} \right) \left(\frac{1}{\lambda_2 + \lambda_3}\right)
	\end{multline}
	Using the assumption $\lambda_1 = c_1 \lambda_3$ and $\lambda_2 = c_2 \lambda_3$, we rewrite the previous expression to get:
	\begin{multline}
		\text{Cov}\left(\tau_1, \tau_2\right) =  \frac{2}{\lambda_3^2 \sqrt{c_1 c_2}} \exp \left[ \frac{\left( c_1 +c_2 + 1 \right)^2 }{ 4 c_1 c_2} \right] 
		\int\limits_{\frac{c_1+c_2+1}{2\sqrt{c_1 c_2}} }^\infty  \left[  \frac{1}{ 2 \sqrt{c_1 c_2} u + c_2 - c_1 + 1 }  \right. \\
		\left. + \frac{1}{ 2 \sqrt{c_1 c_2} u + c_1 -c_2 + 1}  \right]e^{-u^2} du - \frac{1}{\lambda_3^2} \left(\frac{1}{c_1 +1 }\right) \left(\frac{1}{c_2 +1 }\right)
	\end{multline}
	As $u \geq \frac{c_1+c_2+1}{2\sqrt{c_1 c_2}}  $, we have that $\frac{1}{ 2 \sqrt{c_1 c_2} u + c_2 - c_1 + 1 }  \leq \frac{1}{2 \left( c_2 +1 \right)}$ and $\frac{1}{ 2 \sqrt{c_1 c_2} u + c_1 -c_2 + 1}  \leq \frac{1}{2 \left(c_1 +1\right)}$. Then:
	\begin{multline}
		\text{Cov}\left( \tau_1, \tau_2 \right) \leq \frac{1}{\lambda_3^2 \sqrt{c_1 c_2}} \left(  \frac{1}{  c_2 +1 }  + \frac{1}{  c_1 +1 } \right) \exp \left[ \frac{\left( c_1 +c_2 + 1 \right)^2 }{ 4 c_1 c_2} \right] 
		\int\limits_{\frac{c_1+c_2+1}{2\sqrt{c_1 c_2}} }^\infty  e^{-u^2} du \\
		- \frac{1}{\lambda_3^2} \left(\frac{1}{c_1 +1 }\right) \left(\frac{1}{c_2 +1 }\right) \\
		= \frac{1}{\lambda_3^2 \left(c
			_1+1\right) \left(c_2+1 \right)}\left[\frac{1}{\sqrt{c_1 c_2}} \left(c_1 +c_2 +2\right) \exp \left[ \frac{\left( c_1 +c_2 + 1 \right)^2 }{ 4 c_1 c_2} \right] \int\limits_{\frac{c_1+c_2+1}{2\sqrt{c_1 c_2}} }^\infty  e^{-u^2} du - 1 \right]
	\end{multline}
	Hence, to get a negative covariance, it suffices to show:
	\begin{equation} \label{Gumbel.NegCov.Proof.FinalStep}
		\frac{1}{\sqrt{c_1 c_2}} \left(c_1 +c_2 +2\right) \exp \left[ \frac{\left( c_1 +c_2 + 1 \right)^2 }{ 4 c_1 c_2} \right] \int\limits_{\frac{c_1+c_2+1}{2\sqrt{c_1 c_2}} }^\infty  e^{-u^2} du < 1
	\end{equation}
	Set $k := \frac{c_1 +c_2 +1}{\sqrt{c_1 c_2}}$ and note that our initial assumption implies that $\frac{1}{\sqrt{c_1 c_2}} \leq \frac{5}{4} \left( \frac{\sqrt{c _1 c_2}}{c_1 + c_2 +1}\right) = \frac{5}{4k} $. Hence:
	\begin{equation} \label{Bound.Erfc}
		\frac{1}{\sqrt{c_1 c_2}} \left(c_1 +c_2 +2\right) \exp \left[ \frac{\left( c_1 +c_2 + 1 \right)^2 }{ 4 c_1 c_2} \right] \int\limits_{\frac{c_1+c_2+1}{2\sqrt{c_1 c_2}} }^\infty  e^{-u^2} du  \leq \left(k + \frac{5}{4k}\right) e^{\frac{1}{4} k^2} \int\limits_{\frac{k}{2}}^\infty e^{-u^2} du
	\end{equation}
	By the proof in the \hyperref[Appendix]{Appendix}, one can see that the previous integral is always smaller than one when $k > 2 $, which is the case here because:
	\begin{equation*}
		\left(\sqrt{c_1} - \sqrt{c_2}\right)^2 +1 >0 \implies c_1 -2 \sqrt{c_1 c_2} +c_2 +1 > 0  \implies c_1+c_2 +1 > 2 \sqrt{c_1 c_2} 
	\end{equation*}
	Hence:
	\begin{equation}
		k = \frac{c_1+c_2+1}{\sqrt{c_1 c_2}} > \frac{2 \sqrt{c_1 c_2}}{\sqrt{c_1 c_2}} > 2
	\end{equation}
	This shows that equation \eqref{Gumbel.NegCov.Proof.FinalStep} holds and consequently, we have  a negative covariance.
\end{proof}

\begin{remark}
	If $c:=c_1=c_2$, the condition of Proposition \ref{Proposition.Neg.Cov} reduces to $5c^2-8c-4\geq 0$ which is equivalent to $c\geq 2$. The interpretation of this is that to get a negative covariance when $\lambda_1 = \lambda_2$, it suffices to have $\lambda_1, \lambda_2$ sufficiently large compared to $\lambda_3$.
\end{remark}

\section{APPLICATIONS} \label{Sect.Application}
\subsection{Application to Epidemiology} \label{Subsect.Application.Epidemiology}
Suppose we are interested in knowing what is the probability of $n$ people getting infected with COVID-19 at the exact same time. The time to infection of each person can be modeled as a stopping time. Some current models (see Britton and Pardoux \cite{Britton2019}) assume independence of these stopping times and thus the probability of them being equal is $0$. However, using our model we can weaken the independence assumption and conclude that:
\begin{theorem}
	If $\left(\tau_1, \tau_2, \dots, \tau_n \right)$ follows the joint distribution described in Section \ref{Sect.Generalization.K.ST}, then:
	\begin{align}
		\Prob\left( \tau_1 = \tau_2 = \dots = \tau_n \right)  = & \E\left[ \int_0^\infty \alpha^{(1,2,\dots, n)}(u) e^{-\sum\limits_{k=1}^n\sum\limits_{j\in\mathcal{C}^n_k} A^j_u}du \right]
	\end{align}
	The innermost sum in the exponent is taken over all the possible combinations $\mathcal{C}^n_k$ of $\binom{n}{k}$. For example, if $k=3$, $j$ could be $(1,3,5)$, $(2,3,n)$, etc; if $k=n$, $j$ can only be $(1,2,\dots, n)$
\end{theorem}

\begin{remark}
	To be more specific if $n=3$, then we get,
	\begin{align}
		\Prob\left( \tau_1 = \tau_2 = \tau_3 \right)  = & \E \Biggl[ \int_0^\infty \alpha^{(1,2,3)}(u) \exp \left[-A^1_u-A^2_u-A^3_u \right.  \nonumber \\
		&  \left.  -A^{(1,2)}_u-A^{(1,3)}_u - A^{(2, 3)}_u -A^{(1,2,3)}_u   \right] \Biggr] du
	\end{align}
\end{remark}

\begin{proof}
	Using the notation from Section \ref{Sect.Generalization}, events in the process $\Lambda_{(1,2,\dots, n)}$ are shocks to the $n$ components of the system. Hence 
	\begin{multline*}
		\left\{ \tau_1=\tau_2=\dots=\tau_{n-1}=\tau_n\right\} =\\
		\left\{  \Lambda_{(1,2,\dots\ n)} \leq \min \left( \Lambda_1, \dots, \Lambda_n, \Lambda_{(1,2)}, \dots, \Lambda_{(n-1,n)}, \Lambda_{(1,2,3)}, \dots, \Lambda_{(1,2,\dots,n-1)}\right)  \right\}
	\end{multline*}
	Then, we find the distribution of:
	$$M:=\min \left( \Lambda_1, \dots, \Lambda_n, \Lambda_{(1,2)}, \dots, \Lambda_{(n-1,n)}, \Lambda_{(1,2,3)}, \dots, \Lambda_{(1,2,\dots,n-1)}\right) $$ 
	under the measure $\tilde{\Prob}(\cdot)$ which stands for $\Prob\left(\cdot | (X_s)_{s\geq 0} \right)$
	\begin{multline}
		\tilde{\mathbb{P}}\left[ M  > t \right] = \tilde{\mathbb{P}}\left[ \Lambda_1>t, \dots, \Lambda_n>t, \dots, \Lambda_{(n,n)}>t, \Lambda_{(1,2,3)}>t, \dots, \Lambda_{(1,2,\dots,n-1)}>t\right] \\
		= \exp\left[-A^1_t-\dots-A^n_t - A^{(1,2)}_t-\dots - A^{(n-1,n)}_t - A^{(1,2,3)}_t - \dots  -A^{(1,2,\dots, n-1)}_t \right]
	\end{multline}
	Hence, $M$, given $\left(X_{s}\right)_{s\geq 0}$, has a continuous distribution with density equal to:
	\begin{equation}
		f_{M}(t) = \left[\sum\limits_{k=1}^{n-1}\sum\limits_{j\in\mathcal{C}^n_k} \alpha^j_t\right] e^{-\sum\limits_{k=1}^{n-1}\sum\limits_{j\in\mathcal{C}^n_k} \alpha^j_t A^{j}_t} 
	\end{equation}
	The innermost sums are taken over all the possible combinations $\mathcal{C}^n_k$ of $\binom{n}{k}$. 
	Then,
	\begin{align}
		\tilde{\mathbb{P}}\left[\Lambda_{1,2,\dots, n} < M \right]& = \int_0^\infty\int_x^\infty\alpha^{(1,2, \dots, n)}_xe^{-A^{(1,2,\dots,n)}_x}f_{M}(y)\,dy\,dx \nonumber \\
		&= \int_0^\infty \alpha^{(1,2,\dots, n)}_u e^{-\sum\limits_{k=1}^n\sum\limits_{j\in\mathcal{C}^n_k} A^j_u}\,du
	\end{align}
	The result follows by taking the expectation.
\end{proof}

\begin{corollary}
	If $\alpha^i(X_t)=\alpha(X_t)$ for all $i=1,2,\dots, n$; $\alpha^{(i,j)}(X_t)=\frac{1}{2}\alpha(X_t)$ for all size 2 combinations in $\binom{n}{2}$; $\alpha^{(i,j,k)}(X_t)=\frac{1}{3}\alpha(X_t)$ for all size 3 combinations in $\binom{n}{3}$; $\dots$; $\alpha^{(1,2,\dots, n)}(X_t)=\frac{1}{n}\alpha(X_t)$. Then, we have:
	\begin{align}
		\Prob\left(\tau_1=\tau_2=\dots=\tau_n\right) = \frac{1}{n} \left[ \frac{1}{\binom{n}{1}+\frac{1}{2}\binom{n}{2}+\dots+\frac{1}{n}\binom{n}{n}} \right]
	\end{align}
\end{corollary}
\begin{corollary}
	If $\alpha^i(X_t)=\alpha(X_t)$ for all $i=1,2,\dots, n$; $\alpha^{(i,j)}(X_t)=2\alpha(X_t)$ for all size 2 combinations in $\binom{n}{2}$; $\alpha^{(i,j,k)}(X_t)=3\alpha(X_t)$ for all size 3 combinations in $\binom{n}{3}$; $\dots$; $\alpha^{(1,2,\dots, n)}(X_t)=n\alpha(X_t)$. Then, we have:
	\begin{align}
		\Prob\left(\tau_1=\tau_2=\dots=\tau_n\right) & = n \left[ \frac{1}{\binom{n}{1}+2\binom{n}{2}+\dots+n\binom{n}{n}} \right] \nonumber \\
		& = \frac{1}{2^{n-1}}
	\end{align}
\end{corollary}

\subsection{Application to Engineering}
A classic problem in Operations Research, basically studied in queuing theory, is that of a complicated machine. The machine fails if one of its key parts fails. Knowing this, designers create a certain redundancy by doubling key components, so that if one fails, there is a back-up ready to assume its duties. To save money, however, one can have one back-up for two components. Suppose $\eta_1$ is the (random) failure time of one component, and $\eta_2$ is the (again, random) failure time of the second component. If they fail at the same time, then the machine itself will fail, since the solitary back-up cannot replace both components simultaneously. One usually considers such a situation to be unlikely, even very unlikely. If it were to happen, however, we would be interested in $P(\eta_1=\eta_2)$. In the conventional models, $P(\eta_1=\eta_2)$ is zero, since they are each exponential, and conditionally independent. However if we consider a third time, $\eta_3$, and if this third stopping time is the (once again, random) time of an external shock (such as the failure of the air conditioning unit, or a power failure with a surge when the power resumes, etc.), and if we let
\begin{equation}\label{engineering1}
	\tau_1=\eta_1\wedge\eta_3\text{ and }\tau_2=\eta_2\wedge\eta_3
\end{equation}
then we are in the case where $P(\tau_1=\tau_2)=\alpha>0$, and in special cases we can calculate the probability $\alpha$ with precision, and in other, more complicated situations, we can give upper and lower bounds, both, for $\alpha$.

A different kind of example, exemplified by the recent, and quite dramatic example is that of the collapse of the Champlain Towers South, in Surfside, Florida (just north of Miami Beach). The twelve story towers fell at night (1:30AM), and killed 98 people who were in their apartments at the time, and presumably even in their beds. The towers had a pool deck above a covered garage. The towers holding up the pool deck were too thin, and not strong enough to withstand the stresses imposed on them over four decades. Water was seen pouring into the parking garage only minutes before the collapse. 

In the main building that collapsed, structural columns were too narrow to accommodate enough rebar, meaning that contractors had to choose between cramming extra steel into a too-small column (which can create air pockets that
accelerate corrosion) or inadequately attaching  floor slabs to
their supports. Our model would have several stopping times, each one
representing the failure of a different component of the structure. Choosing two important ones, such as:
(1) The corrosion of the rebar supports within the concrete, due
to the salt air and massive strains due to violent weather,
which plagues the Florida coast during hurricane season; (2) The use of a low quality grade of concrete, violating
regulations of the local government in the construction of the 
towers, leading to concrete integrity decay due to 40 years of
seaside weather.

One way to model this is to take a vector of two Cox
constructions, using two independent exponentials $Z_1$ and $Z_2$
to construct our failure times $\tau_1$ and $\tau_2$. 
This then gives us that $\tau_1$ and $\tau_2$ are conditionally
independent, given the underlying filtration $\mathbb{F}=(\mathcal{F}_t)_{t\geq 0}$. This leads to $P(\tau_1=\tau_2)=0$, even if
they have the exact same compensators.

Other models have been proposed, such as the general
framework presented in the book of Anna Aksamit and
Monique Jeanblanc \cite{Aksamit.Jeanblanc}, where the random variables $Z_1$ and $Z_2$
are multivariate exponentials, but with a joint density describing
how they relate to each other. However, even in this more general setting, we have
$P(\tau_1=\tau_2)=0$.

Assuming it is the stopping times occurring simultaneously
that causes the collapse, we want a model that allows us to
have $P(\tau_1=\tau_2)>0$. 
Let's assume we have three standard Cox Constructions, with
independent exponentials $(Z_1; Z_2; Z_3)$, with different
compensators ($\int_0^t\alpha^i_sds $ for $ i=1,2,3$).

Call the three
stopping times $\eta_1,\eta_2,\eta_3$.
The time $\eta_3$ could be anything, such as a hurricane putting
heavy stress on the building, or an earthquake, or some other
external factor. (Current forensic analysis, which is ongoing as we write this, suggests that the collapse of the pool-deck created a seismic shock sufficient to precipitate the collapse of the south tower, and weakened the structural integrity of the north tower to such an extent that it was condemned and then deliberately destroyed.) However in this case we can take $\eta_3$ to be the
time of the flooding of the parking structure under the swimming pool and pool deck in general.

Simulations commissioned by the newspaper \emph{The Washington Post} and done by a team led by Khalid M. Mosalam of the University of California, Berkeley, show how that might have
happened and indicate that it is a plausible scenario (Swaine et al. \cite{Swaine}).

As in \eqref{engineering1}, we define
\begin{equation}\label{engineering2}
	\tau_1=\eta_1\wedge\eta_3\text{ and }\tau_2=\eta_2\wedge\eta_3,
\end{equation}
where we take the stopping time $\eta_3$ to be the time of the collapse of the pool deck.  As the noted engineer H. Petroski has pointed out \cite{Petroski} it is often the case that multiple things happen at once in order to precipitate a disaster such as the fall of the Champlain Tower South. Indeed, the forensic engineer R. Leon of Virginia Tech is quoted in \emph{American Society of Civil Engineers}, 2021 \cite{BW}: ``I think it is way too early to tell,"said Roberto Leon, P.E., F.SEI, Dist.M.ASCE, the D.H. Burrows Professor of Construction Engineering in the Charles Edward Via Jr. Department of Civil and Environmental Engineering at Virginia Tech.``It's going to require a very careful forensic approach here, because I don't think the building collapsed just because of one reason. What we tend to find in forensic investigations is that three or four things have to happen for a collapse to occur that is so catastrophic.''

Professor Leon is a widely respected authority in forensic civil engineering, and the key insight for us is his last statement that three or four things have to happen simultaneously for a catastrophic collapse. 

Less dramatic examples, but quite pertinent to the recent attention being paid to the decay of infrastructure around the US, provide more examples of the utility of this approach. A first example is the Interstate 10 Twin Span Bridge over Lake Pontchartrain north of New Orleans, LA. It was rendered completely unusable by Hurricane Katrina, but the naive explanation was demonstrated false by the fact that several other bridges that had the same structural design remained intact. Upon investigation, it was determined  that air trapped beneath the deck of the Interstate 10 bridges was a major contributing factor to the bridge's collapse. While major, it was not the only contributing factor (Chen et al. \cite{GC}).

A final example is the derailment of an Amtrak train near Joplin, Montana in September, 2021. 154 people were on board the train, and 44 passengers and crew were taken to area hospitals with injuries. The train was traveling at between 75 and 78 mph, just below the speed limit of 79 mph on that section of track when its emergency brakes were activated. The two locomotives and two railcars remained on the rails and eight cars derailed. Investigations of these types of events take years, but preliminary speculation is that the accident could have been caused by problems with the railroad or track, such as a rail that buckled under high heat, or the track itself giving way when the train passed over. Both might also be possible, leading to the two stopping times $\tau_1$ and $\tau_2$, and a situation where $P(\tau_1=\tau_2)>0$. See Hanson and Brown \cite{HB}.

\section{APPENDIX} \label{Appendix}
Showing that equation \eqref{Bound.Erfc} is smaller than 1 for any $k>2$ is equivalent to showing the following inequality:
\begin{equation} \label{Ineq.To.Prove}
	\int_x^\infty e^{-u^2} du \leq \frac{8x}{16x^2 +5} e^{-x^2} \text{ for all } x\geq1
\end{equation}
This is equivalent to the problem of finding an upper bound for the complementary error (erfc) function (recall that $\text{ercf}(x)= \frac{2}{\pi} \int_x^\infty e^{-u^2} du$), which has been widely studied in the literature. A classical result is the Chernoff-Rubin bound (see Chernoff \cite{Chernoff}). More recent work can be found in Chiani et al \cite{Chiani.etal}, Karagiannidis and Lioumpas \cite{Karagiannidis.Lioumpas}, and in Tanash and Riihonen \cite{Tanash.Riihonen}. However, to our knowledge, there is no bound of the type we need here.

We will show a slightly tighter bound than the one in \eqref{Bound.Erfc}, namely, we would like to optimize the following bound in $\ell $
\begin{equation} \label{General.Bound}
	\int_x^\infty e^{-u^2} du \leq \frac{2x}{4x^2 +\ell } e^{-x^2} \text{ for all } x\geq1
\end{equation}
When we say optimize, we mean that $\frac{2x}{4x^2 +\ell } e^{-x^2}$ is as close as possible (yet bigger) to $\int_x^\infty e^{-u^2} du$ for all $x\geq 1$.

In other words, define 
\begin{equation}
	h(x, \ell):= \frac{2x}{4x^2 +\ell } e^{-x^2}- \int_x^\infty e^{-u^2} du
\end{equation}
We want to find $\ell^*$ such that:
\begin{enumerate}
	\item $h(x, \ell^*) \leq h(x, \ell)$ for all $x\geq1$ and $\ell \neq \ell^*$.
	\item \label{Second.Item} $h(x, \ell^*)\geq 0$ for all $x\geq 1$ 
\end{enumerate}
To get the first item in the previous list, we need $\ell$ to be as large as possible. This is because $h(x, \ell)$ is decreasing in $\ell$. However, if we wish to have $h(x, \ell)\geq 0$, $\ell$ cannot be indiscriminately large.

To find $\ell^*$, let us first look at $\frac{\partial}{\partial x} h(x, \ell)$.
\begin{equation}
	\frac{\partial}{\partial x} h(x, \ell) = e^{-x^2}\left[- \frac{2\left(4x^2-\ell\right)}{\left(4x^2+\ell\right)^2}-\frac{4x^2}{4x^2+\ell }\right]+1
\end{equation}
We can find that:
\begin{equation}
	\frac{\partial}{\partial x} h(x, \ell)<0 \text{ when } x>\sqrt{\frac{\ell^2 +2 \ell}{8-4\ell }}
\end{equation}
Hence $h(x, \ell)$ is increasing when $x\in\left[1, \sqrt{\frac{\ell^2 +2 \ell}{8-4\ell}} \right)$, $h(x, \ell)$ is maximized when $x= \sqrt{\frac{\ell^2 +2 \ell}{8-4\ell }}$ and it is decreasing when $x\in\left(\sqrt{\frac{\ell^2 +2 \ell}{8-4\ell }}, \infty \right)$.

Also note that
\begin{equation}
	\lim_{x\rightarrow\infty} h(x, \ell) = 0 \text{ for any } \ell \text{ positive} 
\end{equation}

Hence, to satisfy the \hyperref[Second.Item]{second} item in the list above (i.e., $h(x, \ell^*)\geq 0$ for all $x\geq 1$), it suffices to pick $\ell^*$ such that $h(1, \ell^*)=0$, i.e.,
\begin{equation}
	h(1, \ell^*)= \frac{2}{4+\ell^*}e^{-1} - \int_1^\infty e^{-u^2} du  = 0
\end{equation}
Solving for $\ell^*$, we find that:
\begin{equation}
	\ell^* = \frac{2}{e\int_1^\infty e^{-u^2} du} - 4 \approx 1.27935
\end{equation}
We can then conclude that the maximum of $h(x, \ell^*)$ occurs when
\begin{equation}
	x= \sqrt{\frac{(\ell^*)^2 +2 \ell^*}{8-4\ell^*}} \approx 1.2043 \quad \text{Recall that } \ell^* \approx 1.27935
\end{equation}
We can then numerically calculate that
\begin{equation}
	h(1.2043, \ell^*) \approx 0.00131266
\end{equation}
This last equation means that the maximum difference between $\frac{2x}{4x^2 +\ell } e^{-x^2}$ and $ \int_x^\infty e^{-u^2} du$ is $0.00131266$

As $h(x, \ell)$ is decreasing as $\ell$ increases, $h(x,\ell)\geq 0$ for all $x\geq 1$ as long as $\ell < \ell^* \approx 1.27935$. In \eqref{Ineq.To.Prove}, we set $\ell=\frac{5}{4} < \ell^*$ and hence the work above shows that equation \eqref{Bound.Erfc} is indeed smaller than 1.

\bibliographystyle{abbrvnat}
\bibliography{bibliography.bib}
	
\end{document}